\documentclass[12pt]{article}
\usepackage{amsmath,amssymb,amsthm, mathtools, xcolor}
\long\def\symbolfootnote[#1]#2{\begingroup%
\def\thefootnote{\fnsymbol{footnote}}\footnote[#1]{#2}\endgroup}

\newtheorem{thm}{Theorem}[section]
\newtheorem{prop}[thm]{Proposition}
\newtheorem{lem}[thm]{Lemma}
\newtheorem{cor}[thm]{Corollary}

\theoremstyle{definition}

\newtheorem{defn}[thm]{Definition}

\newtheorem{example}[thm]{Example}

\theoremstyle{remark}

\usepackage{comment}

\title{A remark on optimal  data spaces for classical solutions of $\bar\partial$ }

\author{Martino Fassina, Yifei Pan and Yuan Zhang\footnote{partially supported by NSF DMS-1501024}}
\date{}

\begin{document}

\maketitle

\begin{abstract}
We study the minimal regularity required on the datum to guarantee the existence of classical $C^1$ solutions to the inhomogeneous Cauchy-Riemann equations on planar domains.
\end{abstract}
\renewcommand{\thefootnote}{\fnsymbol{footnote}}
\footnotetext{\hspace*{-7mm}
\begin{tabular}{@{}r@{}p{16.5cm}@{}}
& 2010 Mathematics Subject Classification. Primary 45E05; Secondary 45P05.\\
& Keywords and phrases.
$\bar\partial$-equation, weak solutions, Log-continuous spaces.
\end{tabular}}

\section{Introduction}
Let $\Omega$ be a bounded domain in $\mathbb C$ with $  C^{1,\alpha}$ boundary, where $\alpha>0$, and let ${\bf f}$ be a $(0,1)$ form on $\Omega$. 
Consider the Cauchy-Riemann equation
\begin{equation}\label{eqn}
    \bar\partial u =\mathbf f \ \ \text{on}\ \ \Omega.
\end{equation}
The standard singular integral theory (see \cite[Chapter 1]{V}) implies the solvability of \eqref{eqn} in several function spaces. For instance, if $\mathbf f\in L^p(\Omega), 1<p\le 2$, then there exists a weak solution $u\in L^q(\Omega)$, $q<\frac{2p}{2-p}$ to \eqref{eqn}. Moreover, if $\mathbf f\in L^p(\Omega), p>2$, then there exists a weak solution $u\in C^\gamma(\Omega)$, $ \gamma =\frac{p-2}{p}$.  On the other hand, if $\mathbf f\in C^\alpha(\Omega)$ for some $0<\alpha <1$, then (\ref{eqn}) admits a classical (i.e., $C^1(\overline{\Omega})$) solution $u\in C^{1, \alpha}(\Omega)$. The purpose of the note is to study the minimal regularity required on the datum $\mathbf f$ to guarantee the existence of a classical solution to (\ref{eqn}). The following example shows that continuity is not sufficient.

\begin{example}\label{me} Consider the equation $ \bar\partial u = \mathbf f_\nu =f_\nu d\bar z$ on the disk $D(0, \frac{1}{2}):=\{z\in \mathbb C: |z|<\frac{1}{2}\}$, where $\nu>0$ is fixed and \begin{equation}\label{12}
    f_\nu(z): =  \left\{
      \begin{array}{cc}
     \frac{z }{\bar z\ln^\nu |z|^2} & z\ne 0\\
     0 & z=0.\end{array}\right.
\end{equation}  Clearly, $\mathbf f_\nu\in C(\overline{D(0, \frac{1}{2})})$. However (see the proof at the end of the paper) if $\nu\le 1$,  then there exists no solution $u\in C^1(\overline{D(0, \frac{1}{2})})$.
\end{example}

We consider subspaces of $C(\overline{\Omega})$ constisting of functions satisfying a logarithmic continuity condition.
\begin{defn}
Let $\Omega$ be a bounded domain in $\mathbb R^n, $ $k\in \mathbb Z^+\cup\{0\}, \nu\in \mathbb R^+$. A  function $f\in C^k(\Omega)$  is said to be in $ C^{k, Log^\nu L}(\Omega)$ if
$$\|f\|_{C^{k, Log^\nu L} (\Omega)}:  =\sum_{|\gamma|=0}^k\sup_{w\in \Omega}|D^\gamma f(w)|+ \sum_{|\gamma|=k}\sup_{w, w+h\in \Omega} |D^\gamma f(w+h) - D^\gamma f(w)| |\ln |h||^{\nu}<\infty.$$
A $(0,1)$ form $\mathbf f$ is said to be in $ C^{k, Log^\nu L}(\Omega)$ if all its components are in $ C^{k, Log^\nu L}(\Omega)$.
\end{defn}

When $k=0$, the space $C^{0, Log^\nu L} (\Omega)$ is abbreviated as $C^{Log^\nu L} (\Omega)$. For any $k\in \mathbb Z^+\cup\{0\}, 0<\nu<\mu$, and $0<\alpha< 1$, we have \[ C^{k, \alpha}(\Omega)\xhookrightarrow{}  C^{k,   Log^{\mu}L} (\Omega) \xhookrightarrow{}   C^{k,   Log^{\nu}L} (\Omega) \xhookrightarrow{}   C^k (\overline{\Omega}), \] where every inclusion map is a continuous embedding. In our main result we prove, in particular, that a classical solution to \eqref{eqn} exists whenever ${\bf f}$ is in $C^{ Log^{\nu}L} (\Omega)$ for some $\nu>1$. Here is the precise statement.
\begin{thm}\label{mt2}
Let $\Omega\subset \mathbb C$ be a  bounded domain  with $C^{1, \alpha}$ boundary, where $\alpha>0$. Assume that $\mathbf f=fd\bar z\in C^{Log^\nu L}(\Omega), \nu>1$. Then there exists a   solution $u\in C^{1, Log^{\nu-1}L}(\Omega)$ to  $\bar\partial u =\mathbf f$
 such that   $\|u\|_{C^{1, Log^{\nu-1}L}(\Omega)}\le C\|\mathbf f\|_{C^{Log^{\nu}L}(\Omega)}$, where $C$ depends only on $\Omega$ and $\nu$. In particular, $u\in C^1(\overline{\Omega})$, with $\|u\|_{C^{1}(\overline{\Omega})}\le C\|\mathbf f\|_{C^{Log^{\nu}L}(\Omega)}$.
\end{thm}

Example \ref{me} shows that the assumption $\nu>1$ in Theorem \ref{mt2} cannot be relaxed. In this sense, Theorem \ref{mt2} identifies the largest possible data set guaranteeing the existence of classical solutions to \eqref{eqn}.

The following example shows another way in which Theorem \ref{mt2} is sharp: the loss of 1 in the order of Log-continuity of the solution is optimal.

\begin{example}\label{16}
Let $\nu>0$ be fixed, and consider the equation $ \bar\partial u = \mathbf f_\nu =f_\nu d\bar z $  on the disk $D(0, \frac{1}{2})$, where $f_\nu$ is defined by (\ref{12}). Then $\mathbf f_\nu\in C^{Log^\nu L}({D(0,\frac{1}{2})})$. However (see the proof at the end of the paper) if $\nu>1$, there does not exist a weak solution in $ C^{1, Log^\mu L}({D(0,\frac{1}{2})})$ for any $\mu> \nu-1$.
\end{example}

\section{Preliminaries on the integral operators $T$ and $^2T$}

Let $\Omega$ be a  bounded domain in $\mathbb C$ with $C^{1, \alpha}$ boundary, where $\alpha>0$. Given a function $f\in C(\overline{\Omega})$, define \begin{equation*}
\begin{split}
 Tf(z): =\frac{-1}{2\pi i}\int_\Omega \frac{f(\zeta)}{\zeta- z}d\bar{\zeta}\wedge d\zeta,\,\,\ \ z\in \Omega.
\end{split}
\end{equation*}
It is well known that $T$ is a solution operator to $\bar\partial$ on planar domains in several function spaces (see for instance \cite{AIM} and \cite{V}). For $f\in C(\overline{\Omega})$, we have \begin{equation}\label{V}
    \frac{\partial}{\partial \bar z} Tf ={ f},\,\,\,\quad\quad\frac{\partial}{\partial  z} Tf = p.v.\frac{-1}{2\pi i}\int_\Omega \frac{f(\zeta)}{(\zeta-\cdot)^2}d\bar\zeta\wedge d\zeta =: Hf\,\, 
\end{equation}
in $\Omega$ in the sense of distributions \cite[Theorem 1.32]{V}. Here {\it p.v.} represents the principal value.


In their inspiring paper \cite{NW}, Nijenhuis and Woolf introduced the related integral operator $^2T$. For functions $f\in C^\alpha(\Omega)$, where $0<\alpha<1$, define $$ ^2Tf(z) :=  \frac{-1}{2\pi i}\int_\Omega \frac{f(\zeta)-f(z)}{(\zeta-z)^2}d\bar\zeta\wedge d \zeta,\ \ z\in \Omega.$$
$^2T$  is a bounded operator from the space $C^\alpha(\Omega)$ to itself whenever $0<\alpha<1$ (see \cite[Appendix 6.1.e]{NW} for a proof in the case of $\Omega$ being a disk, and \cite[Theorem 1.32]{V} for the general case). The next proposition shows that $^2Tf$ is well defined also for functions $f$ in Log-continuous spaces, and describes the connection between the integral operators $H$ and ${}^2T$.

\begin{prop}\label{13}
Let $\Omega\subset\mathbb{C}$ be a bounded domain with $C^{1, \alpha}$ boundary, where $\alpha>0$. For every $f\in C^{Log^\nu L}(\Omega), $ with $\nu>1$, the function $^2Tf $ is well defined in $\Omega$. Moreover, letting
\begin{equation}\label{14}
    \Phi(z): = \frac{1}{2\pi i}\int_{b\Omega} \frac{d\bar\zeta}{\zeta - z},\,\,\,\,\,\,z\in\Omega.
\end{equation}
we have
\begin{equation}\label{5}
Hf(z) =   {}^2Tf(z) - f(z)\Phi(z),\,\,\,  z\in\Omega.
\end{equation}
 In the special case of $\Omega$ being a disk centered at 0, then
\begin{equation}\label{6}
Hf(z) =    {}^2Tf(z),\,\,\, z\in\Omega.
\end{equation}
\end{prop}

To prove Proposition \ref{13} we need the two elementary lemmas below. Throughout the paper, unless otherwise specified, we use $C$ to represent a positive constant which depends only on $\Omega$ or $\nu$, and which may be different at each occurrence.

\begin{lem}\label{gi}
Let $\Omega\subset\mathbb C$ be a bounded domain with $C^{1, \alpha}$ boundary, where $\alpha>0$, and let $\Phi$ be defined as in \eqref{14}.
Then $\Phi\in C^\alpha(\Omega)$. Moreover, if $\Omega$ is a disk centered at 0, then $ \Phi \equiv 0$ on $\Omega$.
\end{lem}
\begin{proof}

By Stokes' Theorem,
$$\int_{b\Omega} \frac{d\bar\zeta}{\zeta - z} = \int_{b\Omega} \frac{\bar\zeta d\zeta}{(\zeta - z)^2} = \left(- \int_{b\Omega} \frac{\bar\zeta d\zeta}{\zeta - z}\right)' = (S\bar z)',  $$
where the operator $S$ is defined for functions $f\in C(\overline{\Omega})$ by
$$ Sf(z):=-\int_{b\Omega} \frac{f(\zeta)}{\zeta-z}d\zeta, \quad z\in\Omega. $$
Recall that  $S$
sends the space $C^{1, \alpha}(\Omega)$ into itself, with $\|Sf\|_{C^{1, \alpha}(\Omega)}\le C\|f\|_{C^{1, \alpha}(\Omega)}$ for some constant $C$ depending only on $\Omega$ (see \cite[Theorem 1.10]{V}). Hence $\Phi\in C^\alpha(\Omega)$.

If $\Omega\subset\mathbb{C}$ is a disk centered at 0 with radius $R$, then using the relation $\bar\zeta = R^2/\zeta$ on $b\Omega$ and the Residue Theorem, we get
 $$ \Phi(z) = -\frac{1}{2\pi i} \int_{b\Omega} \frac{R^2d\zeta}{\zeta^2(\zeta - z)} = \left\{
      \begin{array}{cc}
   R^2 (\frac{1}{z^2}-\frac{1}{z^2})  &\text{if}\  z\in \Omega, z\ne 0 \\
    0 &\text{if}\  z=0.
    \end{array}
\right.    $$
Hence $\Phi\equiv 0$ in $\Omega$.
\end{proof}

\begin{lem}\label{elem}
Let  $\nu\in \mathbb R^+$. There exists a constant $C$ depending only on $\nu$ such that, for every choice of $h,h_0$ with $0<h\le h_0<1$, the following hold: \\
\begin{enumerate}
\item $\int_h^{h_0}  s^{-2}|\ln  s|^{-\nu} d s\le
    h^{-1}|\ln h|^{-\nu}. $
\item If $\nu>1$, then $\int_0^h s^{-1}|\ln s|^{-\nu} ds\le C|\ln h|^{1-\nu}$.
\end{enumerate}

\end{lem}

\begin{proof}
\begin{enumerate}
\item Integration by parts yields, when  $\nu>0$ and  $0<h\le h_0<1$,
$$\int_h^{h_0} s^{-2}|\ln s|^{-\nu} ds =    h^{-1}|\ln h|^{-\nu} -h_0^{-1}|\ln h_0|^{-\nu}-\nu \int_h^{h_0} s^{-2}|\ln s|^{-\nu-1} ds.$$
In particular,
$$\int_h^{h_0} s^{ -2}|\ln s|^{-\nu} ds \le h^{ -1}|\ln h|^{-\nu}.
$$
\item Direct integration gives, for $\nu>  1$ and $0<h<1$,
$$ \int_0^h s^{-1}|\ln s|^{-\nu} ds= \frac{1}{\nu-1} |\ln h|^{-\nu+1},$$
which proves the second part of the lemma.
\end{enumerate}
\end{proof}

\begin{proof}[Proof of Proposition \ref{13}:]
  Fix $z\in \Omega$, and let $h_0$ be such that $0< h_0<1$. By Lemma \ref{elem} part 2,
\begin{equation}\label{15}
    \begin{split}
|2\pi{}^2Tf(z)| = &\left|\int_\Omega \frac{f(\zeta)-f(z)}{(\zeta-z)^2}d\bar\zeta\wedge d \zeta\right|\\
\le& \left|\int_{D(z, h_0)\cap \Omega}\frac{f(\zeta)-f(z)}{(\zeta-z)^2}d\bar\zeta\wedge d \zeta\right| + \left|\int_{\Omega\setminus D(z,  h_0)}\frac{f(\zeta)-f(z)}{(\zeta-z)^2}d\bar\zeta\wedge d \zeta\right| \\
\le & C\|f\|_{C^{Log^\nu L}(\Omega)}\int_0^{h_0}|\ln s |^{-\nu}s^{-1}ds + Ch_0^{-2}\|f\|_{C(\Omega)}\\
\le& C\|f\|_{C^{Log^\nu L}(\Omega)}|\ln h_0 |^{-\nu+1}\\ \le& C\|f\|_{C^{Log^\nu L}(\Omega)}.
    \end{split}
\end{equation}
In particular, $^2Tf $ is well defined pointwise in $\Omega$.

A direct computation gives
\begin{equation}\label{1}
    \begin{split}
  p.v.\int_\Omega \frac{f(\zeta)}{(\zeta-z)^2}d\bar\zeta\wedge d\zeta  =&\lim_{\epsilon\rightarrow 0 }\int_{\Omega\setminus D(z, \epsilon)} \frac{f(\zeta)}{(\zeta-z)^2}d\bar\zeta\wedge d\zeta \\
  =& \lim_{\epsilon\rightarrow 0 }\left(\int_{\Omega\setminus D(z, \epsilon)} \frac{f(\zeta)-f(z)}{(\zeta-z)^2}d\bar\zeta\wedge d\zeta + f(z)\int_{\Omega\setminus D(z, \epsilon)} \frac{1}{(\zeta-z)^2}d\bar\zeta\wedge d\zeta \right).
    \end{split}
\end{equation}
Note that
\begin{equation}\label{p}
\left|\frac{f(\zeta)-f(z)}{(\zeta-z)^2}\right|\le C\|f\|_{C^{ Log^\nu L}(\Omega)} |\zeta-z|^{-2}|\ln|\zeta-z||^{-\nu}.
\end{equation}
By Lemma \ref{elem} part 2, the function on the right side of \eqref{p} belongs to $L^1(\Omega)$. Hence the dominated convergence theorem implies
\begin{equation}\label{2}
\lim_{\epsilon\to 0}\int_{\Omega\setminus D(z, \epsilon)} \frac{f(\zeta)-f(z)}{(\zeta-z)^2}d\bar\zeta\wedge d\zeta= -2\pi i{}^2Tf(z).
\end{equation}
On the other hand,
\begin{equation}\label{3}
\begin{split}
    \int_{\Omega\setminus D(z, \epsilon)} \frac{1}{(\zeta-z)^2}d\bar\zeta\wedge d\zeta =&  \int_{b\Omega}\frac{1}{\zeta-z} d\bar \zeta - \int_{bD(z, \epsilon)}\frac{1}{\zeta-z} d\bar \zeta
    = \int_{b\Omega}\frac{1}{\zeta-z} d\bar \zeta
  .
    \end{split}
\end{equation}
Here the first equality makes use of Stokes' Theorem, and the second equality follows from Lemma \ref{gi}. Combining \eqref{1}, \eqref{2}, and \eqref{3}, we conclude
\begin{equation*}
  Hf(z) = {}^2Tf(z) -f(z)\Phi(z),\ \  z\in \Omega,
\end{equation*}
which proves \eqref{5}.

In the case of $\Omega$ being a disk, \eqref{6} follows from \eqref{5} together with Lemma \ref{gi}.
\end{proof}

\section{Optimal bounds for $^2T$ and $T$ in Log-continuous spaces}

In this section we study the boundedness of the operator $^2T$ in the space $C^{Log^\nu L}(\Omega)$. We will show in Theorem \ref{mt} that $^2T$ is a bounded linear operator from $ C^{Log^\nu L}(\Omega)$ into $ C^{Log^{\nu-1} L}(\Omega)$ when $\nu>1$. As a consequence, we derive our Main Theorem \ref{mt2}.

We begin by pointing out that $^2T$ does not send $C(\overline{\Omega})$ into itself, as shown by the following example.

\begin{example}\label{e2}
Let $f$ be the function defined on the disk $D(0,\frac{1}{2})$ by $$f(z) = \left\{
      \begin{array}{cc}
      \frac{z}{\bar z\ln|z|} & z\ne 0\\
     0 & z=0. \end{array}\right.$$ Then $f\in C^{ Log^1 L}(D(0,\frac{1}{2})) \subset C(\overline{D(0,\frac{1}{2})})$. However,
   \begin{equation*}
    \begin{split}
     ^2Tf(0) & =   \int_{D(0,\frac{1}{2})} \frac{1}{|\zeta|^2\ln|\zeta|}d\zeta\wedge d\bar \zeta=C\int_0^{\frac{1}{2}}\frac{1}{r\ln r}dr = \infty,
    \end{split}
\end{equation*}
and therefore $^2Tf(0)$ is not defined. In particular, $^2Tf\notin C(\overline{D(0,\frac{1}{2})})$.
\end{example}

The following important inequality is proved in \cite{NW}.

\begin{lem}\cite[Appendix  6.1d]{NW}\label{NW}
Let $z\in D(0, R)$ and $r>0$. Then
\begin{equation}\label{bou}
\left|\int_{D(0, R)\setminus D(z, r)}\frac{1}{(\zeta-z)^2} d\bar\zeta\wedge d\zeta\right|\le 8\pi.
\end{equation}
\end{lem}

Note that the disk $D(z, r)$ may or may not be completely contained in the ambient disk $D(0, R) $. Moreover, the bound in \eqref{bou} is independent of both $R$ and $r$. The proof of Lemma \ref{NW} in \cite{NW} relies on the symmetries of the disk, and  unfortunately does not carry through when the disk $D(0,R)$ is replaced by more general domains. In the next result, using Lemma \ref{gi}, we achieve a generalization of Lemma \ref{NW} to arbitrary smoothly bounded planar domains.

\begin{lem}\label{NW2}
Let $\Omega\subset\mathbb{C}$ be a  bounded domain with $C^{1, \alpha}$ boundary, where $\alpha>0$. Let $z\in \Omega$ and $r>0$. There exists a positive constant $C$ depending only on $\Omega$ such that
$$\left|\int_{\Omega\setminus D(z, r)}\frac{1}{(\zeta-z)^2} d\bar\zeta\wedge d\zeta\right|\le C.  $$
\end{lem}

\begin{proof}Let $R>0$ be such that  $\Omega\subset\subset D(0, R)$. Since $D(0, R)\setminus D (z, r)  = (D(0, R)\setminus \Omega) \sqcup (\Omega\setminus D(z, r)) $, Lemma \ref{NW} yields
\begin{equation}\label{11}
\begin{split}
    \left|\int_{\Omega\setminus D(z, r)}\frac{1}{(\zeta-z)^2} d\bar\zeta\wedge d\zeta\right| =&  \left| \int_{D(0, R)\setminus D(z, r) }\frac{1}{(\zeta-z)^2} d\bar\zeta\wedge d\zeta  - \int_{D(0, R)\setminus \Omega}\frac{1}{(\zeta-z)^2} d\bar\zeta\wedge d\zeta \right|\\
     \le & 8\pi  + \left|\int_{D(0, R)\setminus \Omega}\frac{1}{(\zeta-z)^2} d\bar\zeta\wedge d\zeta \right|.
    \end{split}
\end{equation} By Stokes' Theorem and Lemma \ref{gi},
\begin{equation}\label{22}
    \int_{D(0, R)\setminus \Omega}\frac{1}{(\zeta-z)^2} d\bar\zeta\wedge d\zeta =\int_{bD(0,R)} \frac{1}{\zeta-z}d\bar \zeta -   \int_{b\Omega} \frac{1}{\zeta-z}d\bar \zeta = -   \int_{b\Omega} \frac{1}{\zeta-z}d\bar \zeta\in C^\alpha(\Omega).
\end{equation}
In particular,  $$\left|   \int_{D(0, R)\setminus \Omega}\frac{1}{(\zeta-z)^2} d\bar\zeta\wedge d\zeta  \right|\le C$$ for some constant $C$ independent of $r$. \end{proof}

We are now ready to prove the main theorem of this section, following the ideas in \cite[Appendix, 6.1e]{NW}.

\begin{thm}\label{mt}
Let $\Omega\subset\mathbb{C}$ be a  bounded domain with $C^{1, \alpha}$ boundary, where $\alpha>0$. Then $^2T$ is a bounded linear operator from $ C^{Log^\nu L}(\Omega)$ into $ C^{Log^{\nu-1} L}(\Omega), \nu>1$.
\end{thm}

\begin{proof}
Let $f\in C^{Log^\nu L}(\Omega)$ and $z\in \Omega$. Let $h_0$ be fixed with  $0< h_0<\frac{1}{2}$. By (\ref{15}),
\begin{equation*}
    \begin{split}
|^2Tf(z)| \le C\|f\|_{C^{Log^\nu L}(\Omega)}.
    \end{split}
\end{equation*}

Next, given $z, z+h\in \Omega$, where $|h|\le h_0$, set $D_0: =\Omega\cap D(z, 2|h|)$. Then
\begin{equation*}
    \begin{split}
    &|^2Tf(z) - {}^2Tf(z+h)|\\
    = & \left|\int_\Omega \frac{f(\zeta)-f(z)}{(\zeta-z)^2}d\bar\zeta\wedge d \zeta - \int_\Omega \frac{f(\zeta)-f(z+h)}{(\zeta-z-h)^2}d\bar\zeta\wedge d \zeta \right|\\
    = & \left|\int_{\Omega\setminus D_0} (f(\zeta)-f(z+h))\left[\frac{1}{(\zeta-z)^2} -\frac{1}{(\zeta-z-h)^2}\right] d\bar\zeta\wedge d \zeta - \int_{\Omega\setminus D_0} \frac{f(z)-f(z+h)}{(\zeta-z)^2}d\bar\zeta\wedge d \zeta\right. \\     & \left.+\int_{ D_0} \frac{f(\zeta)-f(z)}{(\zeta-z)^2}d\bar\zeta\wedge d \zeta - \int_{D_0} \frac{f(\zeta)-f(z+h)}{(\zeta-z-h)^2}d\bar\zeta\wedge d \zeta \right| =: |I_1+I_2+I_3+I_4|.
         \end{split}
\end{equation*}

Let $\gamma$ be the segment connecting $z$ and $z+h$. We can rewrite $|I_1|$ as follows:
\begin{equation*}
    \begin{split}
        |I_1| = &2\left|\int_{\Omega\setminus D_0} (f(\zeta)-f(z+h))\left[\int_\gamma\frac{dw}{(\zeta-w)^3} \right] d\bar\zeta\wedge d \zeta\right| \\
         =& 2\left|\int_\gamma dw \int_{\Omega\setminus D_0} \frac{f(\zeta)-f(z+h)}{(\zeta-w)^3}  d\bar\zeta\wedge d \zeta\right|.
    \end{split}
\end{equation*}
Note, if $\zeta\in \Omega\setminus D_0$ and $w\in \gamma$, then $|\zeta -w|\ge |\zeta-z| -|z-w|\ge |h|$. Thus \[|\zeta-z-h|\le |\zeta-w| +|w -z-h|\le |\zeta-w| +|h| \le 2|\zeta-w|. \]
Writing $\zeta =w + se^{i\theta}$ we see, in particular, that $\Omega\setminus D_0\subset \{\zeta\in\mathbb C: |h|< s<2R\}$, where $R$ is the diameter of $\Omega$. By Lemma \ref{elem} part 1 and the fact that $f\in  C^{Log^\nu L}(\Omega)$,
\begin{equation*}
    \begin{split}
        |I_1| \le & C\|f\|_{C^{Log^\nu L}(\Omega)}|h|\left|\int_{|h|}^{2R}|\ln  s|^{-\nu}  s^{-2}  d s \right|\\
        \le & C\|f\|_{C^{Log^\nu L}(\Omega)}|h|\left(\left|\int_{|h|}^{h_0}|\ln  s|^{-\nu}  s^{-2}  d s \right| + \left|\int_{h_0}^{2R}|\ln  s|^{-\nu}  s^{-2}  d s \right|\right)\\
        \le& C \|f\|_{C^{Log^\nu L}(\Omega)}(|\ln |h||^{-\nu} + |h|)\\
        \le& C \|f\|_{C^{Log^\nu L}(\Omega)}|\ln|h||^{-\nu +1}.
    \end{split}
\end{equation*}

The estimate for $|I_2|$ follows directly from Lemma \ref{NW2}, and the estimate of $|I_3|$ is a straightforward consequence of Lemma \ref{elem} part 2, as shown below:
\begin{equation*}
    \begin{split}
        |I_3|\le C\|f\|_{C^{Log^\nu L}(\Omega)}\left|\int_0^{|h|}|\ln  s|^{-\nu}  s^{-1}  d s \right|\le C \|f\|_{C^{Log^\nu L}(\Omega)}|\ln|h||^{-\nu +1}.
    \end{split}
\end{equation*}
 Finally, $I_4$ is estimated in the same way as $I_3$. \end{proof}

Combining Proposition \ref{13}  with Theorem \ref{mt}, we  obtain the next corollary on the solution operator $T$, from which our Main Theorem \ref{mt2} follows immediately.

\begin{cor}\label{mc}
Let $\Omega$ be a  bounded domain in $\mathbb C$  with $C^{1, \alpha}$ boundary for some $\alpha>0$. Then $T$ is a bounded linear operator from $ C^{Log^\nu L}(\Omega)$ into $ C^{1, Log^{\nu-1} L}(\Omega), \nu> 1$.
\end{cor}

\begin{proof}
Let $f\in C^{Log^\nu L}(\Omega)$. Then $Hf\in C^{Log^{\nu-1} L}(\Omega)$ by Proposition \ref{13} and Theorem \ref{mt}. Hence, by \eqref{V}, $Tf$ is a continuous function whose weak derivatives are in $C^{Log^{\nu-1} L}(\Omega)$. Using a standard mollifier argument, we further know that \eqref{V} holds pointwise in $\Omega$, and in particular $Tf\in C^1(\Omega)$. The $ C^{1, Log^{\nu-1} L}(\Omega)$ estimate for $Tf$ is again a consequence of Theorem \ref{mt}.
\end{proof}

 Example \ref{e2} implies that the assumption $\nu>1$ in Theorem \ref{mt} cannot be dropped. Furthermore, the following example shows that the loss by 1 in the order of Log-continuity in Theorem \ref{mt} and  Corollary \ref{mc} is optimal in a very precise sense.
\begin{example}
Let $f_{\nu}$ be defined by (\ref{12}) for some $\nu>1$. Then $f_\nu\in C^{Log^\nu L}({D(0,\frac{1}{2})})$. However,
\begin{enumerate}
  \item $Tf_\nu \notin C^{1, Log^\mu L}({D(0,\frac{1}{2})})$ for any $\mu> \nu-1$.
   \item $^2Tf_\nu\notin C^{Log^\mu L}({D(0,\frac{1}{2})})$ for any $\mu> \nu-1$.
   \end{enumerate}
\end{example}

\begin{proof}
For the first statement, first check that $u_\nu =  \frac{z}{(1-\nu)\ln^{\nu-1} |z|^2} $ satifies $\frac{\partial}{\partial\bar z} u =f_\nu$ on $D(0,\frac{1}{2}))$ . Hence there exists a holomorphic function $h$ on $D(0,\frac{1}{2})$ such that $Tf_\nu = u_\nu+h$. In particular, $Tf_\nu$ has the same regularity at $z=0$ as $u_\nu$, which is not in $C^{1, Log^\mu L}$ near 0 for any $\mu> \nu-1$.
 The second statement is a direct consequence of the first, in view of the identity ${}^2Tf_\nu = \frac{\partial}{\partial z} T f_\nu$ on $D(0,\frac{1}{2})$.\end{proof}

\begin{proof}[Proof of Example \ref{me} and Example \ref{16}:]
One can check that  $$u_\nu(z): =  \left\{
      \begin{array}{cc}
     \frac{z}{(1-\nu)\ln^{\nu-1}|z|^2} & \nu \ne 1\\
     z\ln(|\ln|z|^2|) & \nu = 1 \end{array}\right. $$ is a solution to $\bar\partial u =\mathbf  f_\nu$ on $D(0, \frac{1}{2})$. Let $u$ be any  weak solution  to $\bar\partial u =\mathbf  f_\nu$ on $D(0, \frac{1}{2})$. Then $u$ differs from $u_\nu$ by a holomorphic function, which is always smooth near $0$. However, when $\nu\le 1$, $u_\nu$ fails to be $C^1$ near $0$. Moreover, for any $\mu>\nu-1>0$, $u_\nu$ (and therefore $u$) is not in $C^{1, Log^{\mu}L}$  near $0$. \end{proof}

\noindent Martino Fassina, martino.fassina@univie.ac.at, Fakult\"at f\"ur Mathematik, Universit\"at Wien, Oskar-Morgenstern-Platz 1, 1090 Wien, Austria\medskip

 \noindent Yifei Pan, pan@pfw.edu, Department of Mathematical Sciences, Purdue University Fort Wayne,
Fort Wayne, IN 46805-1499, USA \medskip

\noindent Yuan Zhang, zhangyu@pfw.edu, Department of Mathematical Sciences, Purdue University Fort Wayne,
Fort Wayne, IN 46805-1499, USA

\end{document}